\documentclass[onefignum,onetabnum]{siamonline171218}

\usepackage{lipsum}
\usepackage{amsfonts}
\usepackage{amssymb}
\usepackage{amsmath}
\usepackage{graphicx}
\usepackage{epstopdf}
\usepackage{algorithmic}
\ifpdf
  \DeclareGraphicsExtensions{.eps,.pdf,.png,.jpg}
\else
  \DeclareGraphicsExtensions{.eps}
\fi

\usepackage{bm}
\usepackage{color}

\usepackage[USenglish]{babel}
\usepackage{float}
\usepackage{graphicx}
\graphicspath{ {images/} }

\usepackage{enumitem}
\setlist[enumerate]{leftmargin=.5in}
\setlist[itemize]{leftmargin=.5in}

\newsiamremark{remark}{Remark}
\newsiamremark{hypothesis}{Hypothesis}
\crefname{hypothesis}{Hypothesis}{Hypotheses}
\newsiamthm{claim}{Claim}
\usepackage{pifont}
\headers{On the lifting and reconstruction of nonlinear systems with multiple invariant sets}{S. Pan, and K. Duraisamy}

\title{On the lifting and reconstruction of nonlinear systems with multiple invariant sets\thanks{Submitted to the editors DATE.
\funding{AFOSR}}}

\author{
Shaowu Pan\thanks{Department of Aerospace Engineering, University of Michigan, Ann Arbor, MI (\email{shawnpan@umich.edu}, \email{kdur@umich.edu} )}
\and
Karthik Duraisamy\footnotemark[2]
}

\usepackage{amsopn}

\newcommand{\bx}{x}

\newtheorem{example}{Example}

\ifpdf
\hypersetup{
  pdftitle={On the lifting and reconstruction of nonlinear systems with multiple invariant sets},
  pdfauthor={Shaowu Pan and Karthik Duraisamy}
}
\fi

\begin{document}

\maketitle

% REQUIRED
\begin{abstract}
The Koopman operator provides a linear perspective on non-linear dynamics by focusing on the evolution of observables in an invariant subspace. Observables of interest are typically linearly reconstructed from the Koopman eigenfunctions. Despite the broad use of Koopman operators over the past few years, there exist some misconceptions about the applicability of Koopman operators to dynamical systems with more than one disjoint invariant sets (e.g., basins of attractions from isolated fixed points). In this work, we first provide a simple explanation for the mechanism of linear reconstruction-based Koopman operators of nonlinear systems with multiple disjoint invariant sets. Next, we discuss the use of discrete symmetry among such invariant sets to construct Koopman eigenfunctions in a data efficient manner. Finally, several numerical examples are provided to illustrate the benefits of exploiting symmetry for learning the Koopman operator. 
\end{abstract}

% REQUIRED

% REQUIRED

\section{Introduction}
Originally introduced in 1931 by Bernard Koopman~\cite{koopman1931hamiltonian} in the context of Hamiltonian dynamical system, the Koopman operator has seen resurgent interest in the understanding, prediction, and control of dynamical systems in diverse fields such as fluid dynamics~\cite{mezic2013analysis}, quantum mechanics~\cite{klus2022koopman}, neuroscience~\cite{marrouch2020data}, and power grids~\cite{susuki2016applied}. 
The idea of Koopman operator is to lift the system state of nonlinear dynamics to a much higher dimensional space where the evolution becomes linear. Generally speaking, such a space can be infinite-dimensional. However, from the eyes of practicioners, the most attractive aspect is to find a non-trivial, finite-dimensional Koopman invariant subspace where one can linearly reconstruct the system state of the original non-linear dynamics~\cite{budivsic2012applied,mezic2013analysis}. Due to the formidable difficulty in directly deriving Koopman eigenfunctions from the governing equations, learning Koopman eigenfunctions from data has been a long pursuit over the past decade~\cite{otto2021koopman,bevanda2021koopman}. 
% 

% now introduce the problems of multiple equilibrium
Recall that any linear dynamical system can only contain up to one \emph{asymptotically} stable equilibrium. Meanwhile, nonlinear systems are well-known to have more than one isolated equilibrium. In a broader sense, a natural follow-up question for those nonlinear systems with multiple disjoint invariant sets would be \emph{``How can one find a set of nonlinear functions to reconcile multiple separated equilibria within a framework of linear systems?''} 
%% now talk about brunton's paper that puts shadow on this topic
In fact, Brunton et al.~\cite{brunton2016koopman} argued against the possibility of purely using \emph{continuous} functions to lift nonlinear systems with multiple asymptotically stable equilibria into a finite-dimensional space that spans the system state~\footnote{Note that if ``multiple fixed points'' are considered instead, this statement does not hold. Consider a counter-example $\dot{x} = \begin{bmatrix} -1 & 0 & 0\\
0 & -1 & 0 \\
0 & 0 & 0
\end{bmatrix} x$. A Koopman invariant subspace with identity mapping can be constructed for such system with infinite number of fixed points. We would like to thank Matthew Kvalheim for pointing this out to us. }. 

% talk about relevant literature review
However, empirical evidence from  early data-driven studies~\cite{li2017extended,otto2019linearly,williams2015data} has shown the success of lifting into a linear system for  non-linear systems using  examples such as that of the unforced Duffing oscillator, where there are three separated equilibria in total. 
From the theoretical perspective, a series of independent studies including the first author's thesis have indicated the use of \emph{discontinuous} functions to reconcile this issue~\cite{bakker2019learning,nandanoori2022data,Liu2021,pan2021robust}. Specifically, Bakker et al.~\cite{bakker2019learning} showed an example of piecewise affine discontinuous scalar dynamics where a Koopman invariant subspace can be constructed with state $x$ and the discontinuous functions that indicate the corresponding domains of attraction.
Pan~\cite{pan2021robust} explicitly showed the use of discontinuous indicator functions to stitch compact Koopman invariant subspaces on each domain of attraction in reconciling the dilemma between the empirical success of lifting into a linear system for arbitrary nonlinear systems with multiple attractors, and further discussed the use of symmetry to further minimize the dimension of the lifted system. In simple terms, for nonlinear systems with multiple disjoint invariant sets, one can find Koopman eigenfunctions that linearly span the system state in the \emph{weak} sense. This explains the previous success from the practitioners~\cite{li2017extended,otto2019linearly}. 
Concurrently, Nandanoori et al.~\cite{nandanoori2022data} stitched local Koopman operators from disjoint invariant sets to obtain a global Koopman operator and discussed the role of symmetric invariant sets on Koopman operator. 
Liu et al.~\cite{Liu2021} further strengthened the insight of Brunton et al.~\cite{brunton2016koopman} by rigorously denying the impossibility of one-to-one continuous immersions of nonlinear system with multiple omega-limit sets on a path-connected forward invariant subset. For other discussions on the existence of immersion on nonlinear systems, interested readers are referred to the recent work by Kvalheim and Arathoon~\cite{kvalheim2023linearizability} together with the work from Liu et al.~\cite{Liu2021,liu2023properties}.

In this paper, we first summarize and then extend the results from the first author's thesis~\cite{pan2021robust} to nonlinear systems with disjoint symmetric invariant sets. Furthermore, we perform numerical experiments to demonstrate the benefit of exploiting known symmetries to further improve generalization performance of learned Koopman operators in predictive modeling.
% talk about the outline of this paper
The rest of the paper is organized as follows. In section \ref{sec:background}, we will review the basic concepts of the Koopman operator. In section \ref{sec:multiple_attractor}, we will discuss the use of discontinuous functions to reconstruct the system state in a weak sense.  In section \ref{sec:symmetry}, we will discuss how to exploit the symmetry of dynamical systems in the context of learning Koopman operator. In section \ref{sec:results}, numerical experiments will be shown to demonstrate the benefits of embedding the discussed symmetry in learning Koopman eigenfunctions. Finally, conclusions are drawn in section \ref{sec:conclusion}.

% talk about the motivation
% In 2016, Brunton et al.~\cite{brunton2016koopman} pointed out that no system with multiple fixed points can admit a finite-dimensional Koopman invariant subspace containing the state variables explicitly. It appears plausible to believe that for systems with multiple fixed points, it may be impossible to find a finite-dimensional Koopman invariant subspace such that one can recover the state variables by linear reconstruction. 

% Therefore, an apparent contradiction exists between the common knowledge and empirical data-driven experiments: \emph{can we find Koopman eigenfunctions that spans system state $\bm{x}$?} 
% We hope to shed light on the above issue. Indeed, partitioning of invariant sets using such Koopman eigenfunction has been proposed for measure-preserving system, dissipative systems and hybrid systems~\cite{govindarajan2016operator,mezic1999method,mezic2021koopman}. 

\section{Background on Koopman operators}
\label{sec:background}

Consider the autonomous dynamical system   
\begin{equation}
\label{eq:nldy}
\dot{x}=F(x),
\end{equation}
where $x \in \mathbb{R}^n$ is the state vector, and $F: \mathbb{R}^n \rightarrow \mathbb{R}^n$ is a smooth vector field. The Koopman operator $\mathcal{K}_t: \mathcal{F} \mapsto \mathcal{F}$ considers the dynamics of observables on the measure space $(\mathcal{M}, \mathfrak{F}, \mu)$ where the $\sigma$-algebra $\mathfrak{F}$ is the family of all subsets of $\mathcal{M}$, $\mu$ is the measure, such that for an observable \textcolor{black}{in a proper function space $\mathcal{F}$, e.g., $\mathcal{F} = L^2(\mathcal{M}, \mu)$}, $\forall g \in \mathcal{F}$, $g: \mathcal{M} \mapsto \mathbb{C}$, $\mathcal{K}_t g \triangleq g \circ S_t, $ where $S_t(x_0)\triangleq S(x_0, t)$ is the flow map.  
For a given observable $g$, our focus lies on the minimal \textcolor{black}{Koopman invariant subspace that spans $g$.} 
%$\mathcal{F}_D$ such that $g \in \mathcal{F}_D$. 
Koopman analysis involves identifying the Koopman eigenfunctions that span such a subspace and their corresponding eigenvalues and eigenfunctions~\cite{mezic1999method}.
The eigenfunctions $s(x)$ and eigenvalues $\lambda$ of the Koopman operator satisfy $\mathcal{K}_t s(x) = \lambda s(x)$. 

The most important use of Koopman operator is to linearize nonlinear systems beyond the vicinity of equilibrium from the Hartman-Grobman theorem~\cite{hirsch2012differential}. As first pointed out by Lan and Mezi\'c~\cite{lan2013linearization}, the Koopman operator is closely related to a global version of the Hartman-Grobman theorem, which under mild conditions states the existence of a continuously differentiable mapping $a: \mathbb{R}^{n} \mapsto \mathbb{R}^n$ and $a(x)=x + h(x)$ that linearizes the nonlinear system in Eq.~\ref{eq:nldy} on the entire basin of attraction. Because $h(x)$ is likely \emph{nonlinear}, one has to reconstruct system state $x$ via an inverse of the nonlinear mapping $a(\bx)$. 
Note that this makes the global Hartman-Grobman theorem~\cite{lan2013linearization} different from Koopman operator theory,~\cite{mezic2013analysis} which by definition requires a function space that spans state $\bx$. 
Nonlinear reconstructions, have however, become quite popular in recent years~\cite{takeishi2017learning,otto2019linearly}. This is predicated on the justification that nonlinear reconstruction is more expressive than linear reconstruction and that they can accommodate complex disjoint invariant sets. It is, however, equivalent to abandoning the concept of \emph{Koopman modes} for the observables of interest~\cite{otto2019linearly}. In other words, this is equivalent to removing the identity mapping from the finite dimensional Koopman invariant subspace. Linear reconstruction is still desirable in the context of model-based control~\cite{kaiser2017data,korda2016linear} since it directly connects the non-linear control with linear counterpart. 
However, systems with multiple invariant sets naturally arise in many physical systems from bifurcation~\cite{hirsch2012differential}.

% Li et al.~\cite{li2017extended} considered the augmentation of the Koopman invariant subspace with neural network-trained observables along with the system state $x$ to force linear reconstruction.  
% one could still retain a linear reconstruction representation, albeit with potentially reduced accuracy compared to nonlinear reconstruction~\cite{otto2019linearly}. 

% In contrast, the Koopman operator enables a broader perspective by facilitating the discovery of global linear representations of the underlying nonlinear dynamics. By examining the identity mapping, one can construct a linear system that is topologically conjugate to the nonlinear dynamical system throughout the entire basin of attraction. 

% Despite the considerable promise of the Koopman operator, however, searching nontrivial Koopman eigenfunctions can pose significant challenges. The main difficulty stems from the infinite-dimensional nature of the Koopman operator. Moreover, the development of algorithms and techniques to identify or approximate the observables and eigenfunctions can be computationally demanding. Another complication is the continuous spectrum which typically occurs in the presence of chaotic or irregular behavior in the underlying dynamical system~\cite{mezic2005spectral} and thus we focus on the point spectrum in the present work. Additionally, the use of time delayed observables can be greatly beneficial in Koopman analysis~\cite{arbabi2017ergodic,kamb2020time,pan2020structure} - this aspect is also not considered in the present work. 

In practice, one often seeks eigenfunctions and eigenvalues of the Koopman operator $\mathcal{K}_{\Delta t}$ given a set of uniformly sampled trajectories of the nonlinear dynamical systems. Specifically, we are interested in seeking a finite dimensional Koopman invariant subspace $\mathcal{F}_D \subset \mathcal{F}$  with $D $ linearly independent observations, defined as
\begin{equation}
\mathcal{F}_D \triangleq \textrm{span}\{ \phi_1, \ldots, \phi_D \} \subset \mathcal{F},
\end{equation}
where 
% $\phi_i \in C^1(\mathcal{M}, \mathbb{R})$  and
$\phi_i \in \mathcal{F}$, $i=1,\ldots, D$. 
Correspondingly, the observation vector $\Phi(x)$ is defined as
\begin{equation}{\label{eq:nlob}}
\Phi (x) \triangleq \begin{bmatrix}
\phi_1(x) & \phi_2(x) & \ldots & \phi_D(x)
\end{bmatrix}^\top.
\end{equation}
Given two sequential system states $x_{n} \triangleq x(t_n)$ and $x_{n} \triangleq x(t_n +\Delta t)$, we have 
\begin{equation}
\mathcal{K}_{\Delta t} \Phi(x_n) = \Phi(x_{n+1}) = K \Phi(x_n),
\end{equation}
where $K\in \mathbb{R}^{D\times D}$ is the finite-dimensional representation of Koopman operator $\mathcal{K}_{\Delta t}$. The eigendecomposition of $K$ leads to $K = P \Lambda P^{-1}$, which leads to the definition of eigenfunctions vector $\Psi$, 
\begin{equation}
\Psi(x_{n+1}) = P^{-1} \Phi(x_{n+1}) = \Lambda P^{-1} \Phi(x_n) = \Lambda \Psi(x_{n}).
\end{equation}
Lastly, if we can reconstruct $x$ from $\Phi(x)$, i.e., there exists a matrix $C$ such that $x = C\Phi$, then we can define the Koopman modes that satisfy
\begin{equation}
    x = B \Psi(x).
\end{equation}

It is worthwhile re-emphasizing that while linear reconstructions are desirable in applications such as optimal control~\cite{kaiser2017data}, nonlinear reconstruction can be more expressive and can further reduce the dimension of latent space~\cite{li2017extended,otto2019linearly,lusch2017deep,pan2020physics}.

% Given 
% $\dot{{\bPhi}} =  {\bPhi} {\mathbf{K}}$
% we have, by the chain rule
% \begin{equation}{\label{eq:keyeq}}
% \mathbf{F} \cdot \nabla_{x} {\mathbf{\Phi}} = \mathbf{\Phi} \mathbf{K}.
% \end{equation}
% Additionally, we require $\mathcal{F}_D$ rich enough to recover the system state $x$. 

% \subsection{Reconstruction of system state from Koopman eigenfunctions}

% 

\section{Linear reconstruction for nonlinear systems with multiple disjoint invariant sets}
\label{sec:multiple_attractor}

As previously indicated, it is impossible to linearly reconstruct the system state over the entire phase space with only a finite number of \emph{continuous observables} when the nonlinear system admits multiple asymptotically stable equilibria. 

Here we start with a motivating example to illustrate that, when discontinuous observables are allowed, one can reconstruct the system state in the weak sense. Consider the well-known unforced duffing system:
\begin{align}
\label{eq:duffing_eq}
\dot{x_1} &=  x_2, \\
\dot{x_2} &= -\delta x_2 - x_1(\beta + \alpha x_1^2),
\end{align}
where $\delta = 0.5$, $\beta = -1$, $\alpha = 1$. As shown in Fig.~\ref{fig:lifting_duffing}, the extraction of the third ``indicator" Koopman eigenfunction can be interpreted as  the ability of the learning algorithm to discern initial conditions that result in trajectories converging to different fixed points at final time. In other words, states corresponding to different basins of attraction end up linearly evolving on ``parallel" subspaces.

\begin{figure*}[htbp]
 \centering
 \includegraphics[width=\linewidth]{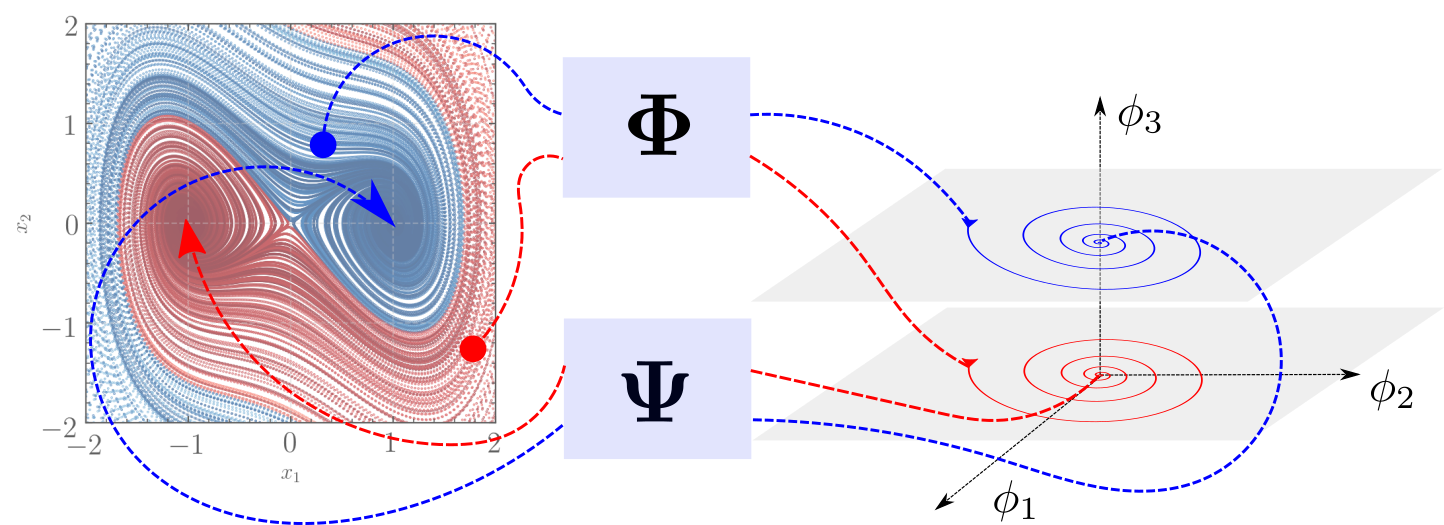}
 \caption{Mechanism of lifting with $\Phi(x)$ into higher dimension for Duffing oscillator, which has two fixed asymptotically stable points (while the third one is unstable at the origin). Note that the above observables $\phi_1,\phi_2,\phi_3$ in the figure is only for the ease of illustration since it is just \emph{one} of the acceptable observables rather than the \emph{one} we obtained in the actual training. Reconstruction is performed via $\Psi(x)$, which can be linear or nonlinear.}
 \label{fig:lifting_duffing}
\end{figure*}

Indeed, the above observation can be extended to the general case with multiple disjoint invariant sets. The overarching idea that has concurrently appeared in Nandanoori et al.~\cite{nandanoori2022data} and the first author's thesis~\cite{pan2021robust} is to prove that the stitched state $\widehat{x} = \sum_{j=1}^{J} \chi_{\mathcal{M}_j} (x)\sum_{l=1}^{D_j}  {w}_{j,l} \phi_{j,l}(x)$ is the same as $x$ almost everywhere. 
In order to achieve this, let's start with the following three assumptions: 
\begin{enumerate}
\item we assume there are $J \in \mathbb{N}$ disjoint invariant sets in $\mathcal{M}$ for the nonlinear system defined in the Eq.~\ref{eq:nldy}. Moreover, for $j=1,\ldots,J$, the union of disjoint invariant sets $\Sigma = \bigcup_{j=1}^{J} \mathcal{M}_j \subset \mathcal{M}$ where $\mathcal{M}_j \subset \mathcal{M}$ and $\Sigma$ has full measure on $\mathcal{M}$\footnote{One could look for an alternate partitioning for measure-preserving continuous systems on a compact domain~\cite{mezic1999method}. Due to Birkhoff's pointwise ergodic theorem, one can guarantee that such a partition covers almost the entire   manifold.}. 

\item we assume there exists a $D_j$-dimensional Koopman invariant subspace spanned by $D_j$ independent real-valued functions $\phi_{j,1},\ldots,\phi_{j,D_j}$, i.e., 
\begin{equation}
    {\mathcal{F}}_{D_j}=\textrm{span}\{{\phi}_{j,1},\ldots, {\phi}_{j,D_j}\},
\end{equation}
with $D_j \ge N$ within any $j$-th invariant set\footnote{Note that if such invariant set is a domain of attraction, then the existence of nontrivial Koopman eigenfunctions has been proved by~\cite{lan2013linearization} under mild conditions.}.

\item we assume that 
\begin{equation}
\forall x \in \mathcal{M}_j, \quad \sum_{l=1}^{D_j} {w}_{j,l} {\phi_{j,l}}({x}) = {x}.
\end{equation}

% Equivalently in the uncentered coordinate space, for the $j$-th basin of attraction, with $\phi_{j,l}(x) = \widetilde{\phi_{j,l}}(x - \mathbf{b}_j)$, one has
% $\mathcal{F}_{D_j}=\textrm{span}\{\phi_{j,1},\ldots, \phi_{j,D_j}\}$, that \emph{linearly reconstructs} the state variable $x$, 
% \begin{equation}
% \forall x \in \mathcal{M}_j, \quad \mathbf{b}_j + \sum_{l=1}^{D_j} \mathbf{w}_{j,l} \phi_{j,l}(x)= x,
% \end{equation}
% where $\mathbf{b}_j$ is some constant vector associated to the attractor, e.g., fixed point location, and $\mathbf{w}_{j,l} \in \mathbb{R}^{N \times 1}$ and $\textrm{rank}(\begin{bmatrix}
% \mathbf{w}_{j,1} & \ldots & \mathbf{w}_{j,D_j}
%  \end{bmatrix}) = N$. 
 \end{enumerate}

With the above three assumptions, we can construct the following observable vector ${\Phi}$ in $\mathcal{F}_{D}$, where $D = \sum_{j=1}^{J} D_j$, 
\begin{equation}
\Phi(x) = 
\begin{bmatrix}
 \chi_{\mathcal{M}_1}(x) \phi_{1,1}(x)
& \ldots & \chi_{\mathcal{M}_J}(x) \phi_{J,D_J}(x)
\end{bmatrix}^\top,
\end{equation}
where $\chi_{\mathcal{M}}: \mathcal{M} \mapsto \{0,1\}$ is the indicator function defined as, 
  \begin{equation}
    \chi_{\mathcal{M}}(x)=
    \begin{cases}
      1, & \text{if}\ x \in \mathcal{M}, \\
      0, & \text{otherwise}.
    \end{cases}
  \end{equation}
%Note that $\chi_{\mathcal{M}_j}$ is also invariant to $\mathcal{K}_t$. %Then it is trivial to see that one can simply use linear mapping to reconstruct the state without resorting to nonlinear reconstruction at all. 
In fact, such an observable vector $\Phi$ induces a weak linear reconstruction of $x$. To see this, we can replace $x$ with $\widehat{x}$, which is a linear combination of each component of ${\Phi}$ almost everywhere, i.e., $\widehat{x} = \sum_{j=1}^{J} \chi_{\mathcal{M}_j} (x)\sum_{l=1}^{D_j}  {w}_{j,l} \phi_{j,l}(x)$, such that we can show that $x$ and $\widehat{x}$ are the same almost everywhere, 
\begin{eqnarray}
& \lVert x - \widehat{x} \rVert^2_{\mathcal{F}} = 
\bigg\lVert x -  \sum_{j=1}^{J} \chi_{\mathcal{M}_j} (x)(\sum_{l=1}^{D_j}  {w}_{j,l} \phi_{j,l}(x)) \bigg\rVert^2_{\mathcal{F}} \nonumber \\
&= \int_{\mathcal{M}} \bigg \lVert x - \sum_{j=1}^{J} \chi_{\mathcal{M}_j}(x) (  \sum_{l=1}^{D_j}  {w}_{j,l} \phi_{j,l}(x)) \bigg \rVert^2_2 d\mu \nonumber \\
&= \sum_{j=1}^J \int_{\mathcal{M}_j} \bigg \lVert x  -   \sum_{l=1}^{D_j}  {w}_{j,l} \phi_{j,l}(x) \bigg \rVert^2_2 d\mu =0.
\end{eqnarray}

To summarize, one can easily linearly reconstruct the system state in the presence of multiple disjoint invariant sets by considering globally \emph{discontinuous} observables, which leads to globally \emph{discontinuous} eigenfunctions~\cite{bakker2019learning}. 

% In fact, if one is not allowed to use \emph{discontinuous} observables, one can not linearly reconstruct the system state under mild conditions for nonlinear dynamics with multiple omege-limit sets~\cite{Liu2021}. 

\section{Role of symmetry}
\label{sec:symmetry}

Here we consider nonlinear dynamics with discrete symmetries, which is well-defined through the concept of equivariant dynamical systems~\cite{field2009symmetry}.
\begin{definition}[$\Gamma$-equivariant]
Given a continuous-time dynamical system, $\dot{x}=F(x)$, $x \in \mathcal{M}$, we call the system is $\Gamma$-equivariant with respect to an action group $\Gamma$ defined on $\mathcal{M}$ if 
\begin{equation}
F(\gamma x(t)) = \gamma F(x(t)), 
\end{equation}
$\forall x \in \mathcal{M}$ and $\gamma \in \Gamma$.
\end{definition}
\begin{remark}
Note that the time discretization of the system also satisfies the equivariance, i.e., if $\{x_i\}_{i=1}^{\infty}$ is a trajectory of the system, so does $\{\gamma x_i \}_{i=1}^{\infty}$.     
\end{remark}

\begin{example}[Unforced Duffing Oscillator]
\label{ex:udo}
The unforced duffing oscillator is a $\mathbb{Z}_2$-equivariant dynamical system. It can be seen from the following equation,
\begin{equation}
    \gamma_\rho F(x) = \begin{bmatrix}
        -x_2 \\
        \delta x_2 + x_1(\beta + \alpha x_1^2)
    \end{bmatrix}
    = F(\gamma_\rho x),
\end{equation}
where $\gamma_\rho \in \Gamma_\rho$ is the matrix representation of the action $\gamma \in \Gamma$, 
\begin{equation}
\gamma_\rho =  \begin{bmatrix}
    -1 & 0 \\
    0 & -1
\end{bmatrix}.
\end{equation}
\end{example}

\begin{remark}
Note that $\Gamma$-equivariance only implies  symmetry of the vector fields without affecting the actual number of invariant sets. For example, the Lorenz system is a $\mathbb{Z}_2$-equivariant while there is only one strange attractor. 
\end{remark}

\begin{remark}
However, if a $\Gamma$-equivariant dynamical system does have multiple disjoint invariant sets, then the symmetry will be also shared across disjoint invariant sets. 
\end{remark}

For $\Gamma$-equivariant dynamical systems with multiple disjoint invariant sets under mild conditions, we can have a Koopman mode decomposition for the entire global phase space without data from the entire phase space. 
\begin{theorem}
\label{thm:1}
Given a $\Gamma$-equivariant discrete-time dynamical system $x_{n+1}=F(x_n)$ defined on a manifold $\mathcal{M}$, $n \in \mathbb{N}$ with $J\in \mathbb{N}$ disjoint invariant sets $\{\mathcal{M}_j\}_{j=1}^{J}$ of which the union has full measure in $\mathcal{M}$, 
if for each set $\mathcal{M}_j$, there exists some $\gamma_j \in \Gamma$ such that $\mathcal{M}_j = \gamma_j\mathcal{M}_1$, and if there exists a finite-dimensional Koopman invariant subspace of functions spanned by observation vector $\Phi(x)$ restricted on $\mathcal{M}_1$ that linearly reconstructs the system state $x$ with measurement matrix $C$, then the $l$-step evolution of given initial condition $x_{n} \in \mathcal{M}$ can be shown almost everywhere as,
\begin{equation}
\label{eq:decompose}
x_{n+l} = \sum_{j=1}^{J} \chi_{\mathcal{M}_1}( \gamma_j^{-1} x_n) \gamma_j C K^l \Phi \left( \sum_{j=1}^{J} \chi_{\mathcal{M}_1}(\gamma_j^{-1}  x_n) \gamma_j^{-1} x_n  \right), 
\end{equation}
where $\chi_{\mathcal{M}_1}$ is the indicator function for $\mathcal{M}_1$ and $K$ is the finite-dimensional representation of Koopman operator given basis $\Phi(x)$. 
\end{theorem}

\begin{proof}
% Due to $\cup_{j=1}^{J} \mathcal{M}_j = \cup_{j=1}^{J} \gamma_j \mathcal{M}_1$ has full measure in $\mathcal{M}$ 
Since the union of disjoint sets $\cup_{j=1}^{J} \mathcal{M}_j $ has full measure in $\mathcal{M}$,  for almost any $x_n \in \mathcal{M}$, there exists a  unique $j' \in \{1,\ldots,J\}$ such that $x_n \in \mathcal{M}_{j'}$. Since for any $\mathcal{M}_j$, there exists $\gamma_j \in \Gamma$ such that $\gamma_j \mathcal{M}_1 = \mathcal{M}_j$,  for the previous $x_n \in \mathcal{M}_{j'}$, we can always find $x'_n \in \mathcal{M}_1$ such that $\gamma_{j'} x'_n = x_n$. 

Since $\Gamma$ is a group, $\gamma_{j'}^{-1} \in \Gamma$  for any $x_n \in \mathcal{M}_{j'}, \gamma_{j'}^{-1} x_n = x_n' \in \mathcal{M}_1$. Note that for any $j'\in \{1,\ldots,J\}$, $\mathcal{M}_{j'}$ is an invariant set. Thus for any $l \in \mathbb{N}$, $x_{n+l} = F^{(l)}(x_n)$ can be mapped by $\gamma_{j'}^{-1}$ to $x_{n+l}' \in \mathcal{M}_1$. Since we have a finite-dimensional Koopman invariant subspace $\mathcal{M}_1$ with the observation vector $\Phi(x)$, measurement matrix $C$, and finite-dimensional representation of the Koopman operator $K$, we can linearize the nonlinear dynamics on $\mathcal{M}_1$ as 
$x_{n+l}' = CK^l \Phi(x_{n}')$. However,  in order to extend this relation to $\cup_{j=1}^{J} \mathcal{M}_j$, we need to know $j'$ a priori. One way is to consider the use of indicator functions for each invariant set. Then for any $x_n \in \cup_{j=1}^{J} \mathcal{M}_j$, we have $\sum_{j=1}^{J}  \chi_{\mathcal{M}_j}(x_n) \gamma_j^{-1} x_n \in \mathcal{M}_1$. 

Hence, $x_{n+l}' = CK^l \Phi(x_{n}') =  CK^l \Phi \left( \sum_{j=1}^{J} \chi_{\mathcal{M}_j}(x_n) \gamma_j^{-1} x_n  \right).$ Similarly, $\sum_{j=1}^{J} \chi_{\mathcal{M}_j}(x_n) \gamma_j x_{n+l}' = x_{n+l}$. Finally, notice that $\chi_{\mathcal{M}_j} = \chi_{\mathcal{M}_1} \circ \gamma_j^{-1} $. Thus, combining these together with the fact that $\cup_{j=1}^{J} \mathcal{M}_j $ has full measure on $\mathcal{M}$, we have the theorem proved. \qed 
% = CK^l \Phi( \gamma_{j'}^{-1} x_{n})
% Hence, for any $l \in \mathbb{N}$, the $l$-step evolution of $x_n$ denoted as $x_{n+l}$ can be mapped to an element in $\mathcal{M}_1$ by
% For any $x' \in \cup_{j=1}^{J} \mathcal{M}_j$, we can find $x \in \mathcal{M}_1$ such that $\gamma' x = x'$. 
\end{proof}

\begin{remark}
The above theorem can be viewed as extension of the linear reconstruction described in section~\ref{sec:multiple_attractor} with symmetry considerations. 
\end{remark}

\begin{remark}
The above equation for $l$-step evolution is not unique. In practice, the domain of $\Phi(x)$ is not restricted on $\mathcal{M}_1$, e.g., polynomial basis. Indeed, if we remove the restriction on the domain of $\Phi(x)$ being $\mathcal{M}_1$, we can have the following alternative equation to Eq.~\ref{eq:decompose}, 
\begin{equation}
\label{eq:decompose_2}
x_{n+l} = \sum_{j=1}^{J} \chi_{\mathcal{M}_1}(\gamma_j^{-1} x_n) \gamma_j C K^l \Phi \left( \gamma_j^{-1} x_n  \right).
\end{equation}
\end{remark}

\begin{remark}
From Eq.~\ref{eq:decompose_2}, the dimension of the lifted function space for a global reconstruction of the nonlinear dynamics is \emph{at most} $JD$ where $D$ is the number of components in $\Phi(x)$. If for any $j\in \{1,\ldots,J\}$, $\gamma_{j}$ commutes with $CK^{l}$, then we have 
\begin{align*}
\sum_{j=1}^{J}  \chi_{\mathcal{M}_1}(\gamma_{j}^{-1} x_n) \gamma_{j} C K^l \Phi \left( \gamma_{j}^{-1} x_n  \right) &= C K^l \tilde{\Phi}(x_n),
\end{align*}
where $\tilde{\Phi}(x_n) \triangleq \sum_{j=1}^{J}  \chi_{\mathcal{M}_1}(\gamma_{j'}^{-1} x_n) \gamma_{j'}  \Phi \left( \gamma_{j'}^{-1} x_n  \right) $. This effectively reduce the dimension of lifted space from $JD$ to $D$.  
\end{remark}

Generally speaking for $\Gamma$-equivariant dynamical systems that do not fit into the above situation (e.g., chaotic Lorenz system with only one attractor), we can simply augment the existing training data by exploiting the fact that if a trajectory $\{x_n\}_{n=1}^{\infty}$ is generated by integrating such $\Gamma$-equivariant dynamics,  then for any $\gamma \in \Gamma$, $\{\gamma x_n\}_{n=1}^{\infty}$ is also a trajectory generated by integrating such dynamics~\cite{nandanoori2022data}. 

In the following section, we will discuss the results induced by the above claims with analytical and numerical examples.

\section{Results}
\label{sec:results}

\subsection{Impact of Symmetry on Linear Reconstruction of the Unforced Duffing Oscillator}

As shown in Fig.~\ref{fig:lifting_duffing}, the unforced duffing oscillator satisfies the conditions in Theorem~\ref{thm:1}. If we choose $\mathcal{M}_1$ to be the  domain of attraction that contains the positive fixed point (then leaves $\mathcal{M}_2$ be the one that contains the left fixed point), now the only question is the number of components $D$ in $\Phi(x)$ for the invariant set $\mathcal{M}_1$. Note that the smallest dimension possible for linear reconstruction would be three, as one constant function is required to center the fixed point to the origin while two other functions correspond to the local linearized dynamics in the vicinity of the fixed point. In practice, the larger number of components in $\Phi(x)$, the better Koopman decomposition one will have, if not overfitted. 

Let's choose the $\gamma_{\rho}$ defined in Example.~\ref{ex:udo}. The key to notice is that $\gamma_\rho$ is effectively a scalar, which commutes with $CK^{l}$. Further expanding Eq.~\ref{eq:decompose_2}, we have 
\begin{align}
x_{n+l} &= \chi_{\mathcal{M}_1}(x_n) CK^l \Phi(x_n) + \chi_{\mathcal{M}_2}(x_n) \gamma_{\rho} CK^l \Phi( \gamma_{\rho}^{-1} x_n), \nonumber \\ 
&= \chi_{\mathcal{M}_1}(x_n) CK^l \Phi(x_n) - \chi_{\mathcal{M}_2}(x_n) CK^l \Phi( -x_n),\\
x_{n+l}  &= 
\begin{cases}
CK^l \Phi(x_n), \text{ if $x_n \in \mathcal{M}_1$}, \\
CK^l (-\Phi( -x_n)), \text{ if $x_n \in \mathcal{M}_2$},
\end{cases}
\end{align}
which can be further rewritten as a Koopman decomposition with only $D$ (instead of 2D) observables, 
\begin{equation}
\label{eq:duffing_compact}
x_{n+l} = CK^l \tilde{\Phi}(x_n),
\end{equation}
where $\tilde{\Phi}(x_n) \triangleq (2\chi_{\mathcal{M}_1}(x_n) - 1) \Phi((2\chi_{\mathcal{M}_1}(x_n) - 1)x_n)$.

\subsection{Exploiting symmetry improves generalization performance}

Inspired by the results presented in this work by far, we want to answer the following questions about exploiting symmetry in learning the Koopman operator: 
\begin{itemize}
    \item Can the efficiency of learning Koopman operators be improved by incorporating known symmetries among disjoint invariant sets, e.g., using Eq.~\ref{eq:decompose_2}? 
    \item Can the  learning process benefit from merely augmenting training data using known symmetries?
\end{itemize}

\subsubsection{Example: symmetry-constrained EDMD for unforced Duffing oscillator}

\begin{figure*}[htbp]
 \centering
 \includegraphics[width=0.6\linewidth]{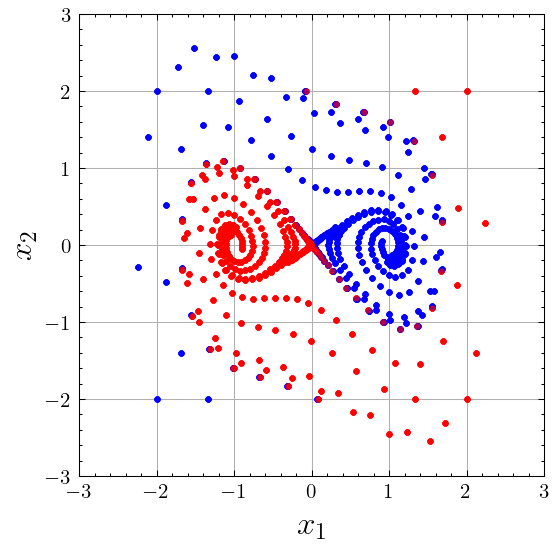}
 \caption{Training data for EDMD on the unforced Duffing oscillator. 49 trajectories for training while blue trajectories end up in the positive equilibrium and red ones in the negative equilibrium.}
 \label{fig:result_2}
\end{figure*}

For the first question, we again take the unforced Duffing oscillator in Eq.~\ref{eq:duffing_eq} as an example. We still choose $\delta=0.5,\beta=-1,\alpha=1$ for the dynamics. As shown in Fig.~\ref{fig:result_2}, here we first collect 49 trajectories with the 
second component of the initial condition $x_2$ being fixed to -2 or 2 with the first component $x_1$  uniformly distributed within -2 and 2 except the third, fourth, and fifth one along $x$-axis being changed to -0.085, -0.08, -0.075 as a refinement near the boundary of domain of attraction. We simulate each trajectory up to time 10 with time interval 0.2.

\begin{figure*}[ht]
 \centering
 \includegraphics[width=\linewidth]{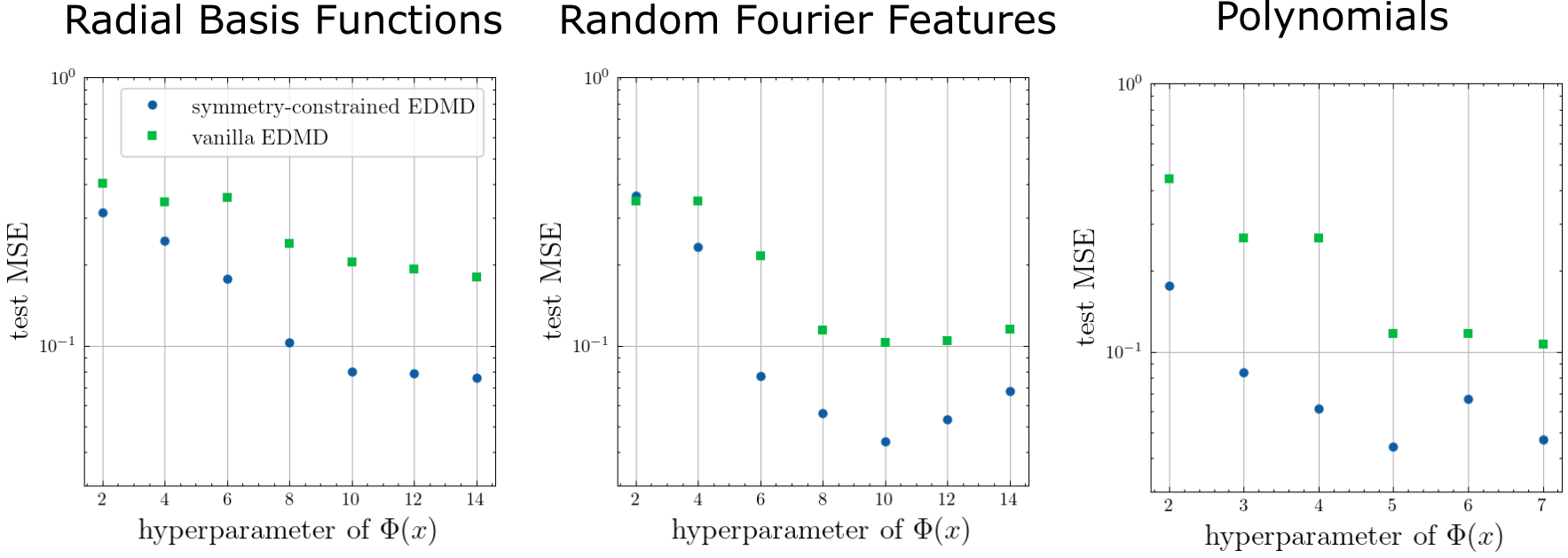}
 \caption{Comparison of generalization performance between vanilla EDMD and symmetry-constrained EDMD over a variety of observables. The comparison is based on averaged mean-squared-error over 100 unseen trajectories with initial condition uniformly distributed within the domain and performed across different choice of observables $\Phi(x)$. Hyperparameter of $\Phi(x)$ from left to right are: number of centers, number of pairs of frequencies, maximum order of polynomials.}
 \label{fig:result_3}
\end{figure*}

To use Theorem~\ref{thm:1},  the indicator function $\mathcal{M}_1$ have to be defined. Hence, we trained a classifier using a simple support vector machine, which serves as an approximation of the indicator function. Given the indicator function, we enforce the symmetry by only performing EDMD on half of the data and use the Eq.~\ref{eq:decompose_2} to extend the model to the entire phase space. We refer to this model as \emph{symmetry-constrained EDMD} in Fig.~\ref{fig:result_3}. For a fair comparison, we also performed EDMD on the entire data from the global phase space that contains \emph{twice} the number of data points compared to the symmetry-constrained case. This is referred to as \emph{vanilla EDMD} in Fig.~\ref{fig:result_3}. As presented in Fig.~\ref{fig:result_3}, the symmetry-constrained model consistently outperforms the one without explicitly enforcing symmetry across a range of hyperparameters for three different types of observable functions. We implemented both models within the \texttt{PyKoopman} framework~\cite{pan2023pykoopman}.

\subsubsection{Example: EDMD trained on symmetry-augmented data for chaotic Lorenz attractor} 

\begin{figure*}[htbp]
 \centering
 \includegraphics[width=\linewidth]{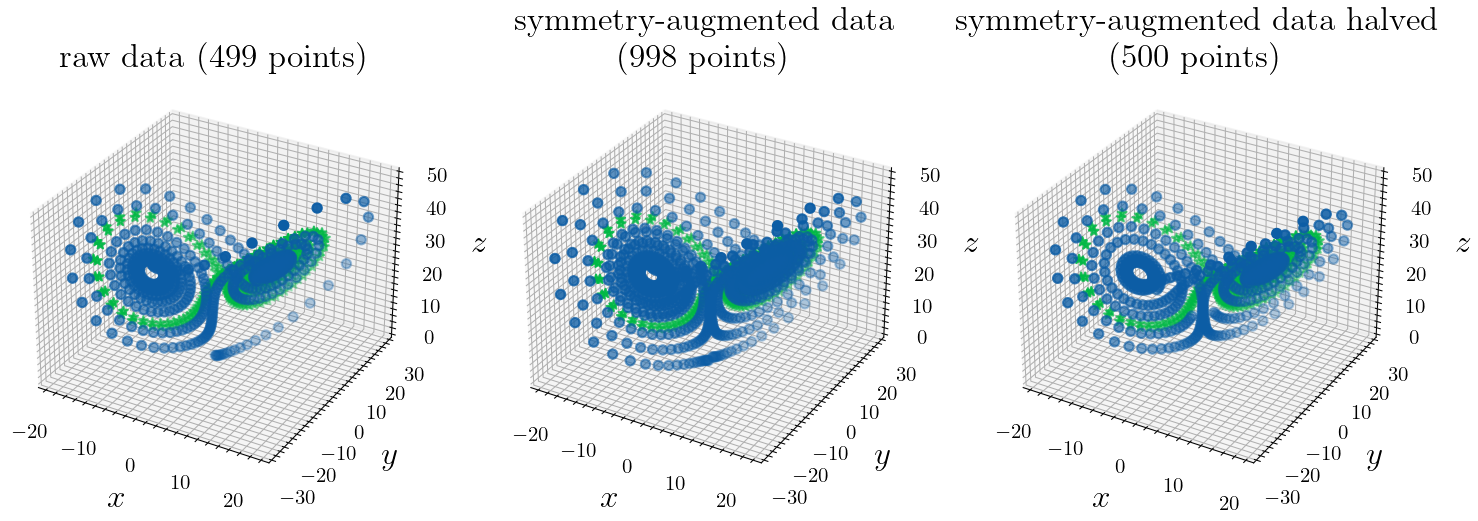}
 \caption{Distributions of three different training data of chaotic Lorenz system projected on $x-y$, $y-z$ and $x-z$ planes. \emph{raw data} refers to the original trajectory collected. \emph{symmetry-augmented data} refers to double the size of the original data by rotating the data $pi$ around $z$ axis. \emph{symmetry-augmented data halved} refers to double the size of halved original data by rotating the data $\pi$ around $z$ axis. The blue dots correspond to the training data in each case. The light green dots are the unseen testing data, which is generated by further integrating the Lorenz system continuing from the last time step of raw data.}
 \label{fig:result_lorenz_1}
\end{figure*}

For the second question, we consider the dynamics of chaotic Lorenz attractor
\begin{align}
\label{eq:lorenz}
\dot{x} &= \sigma(y - x), \\
\dot{y} &= x(\rho - z) - y, \\
\dot{z} &= xy - \beta z.
\end{align}
with $\sigma=10,\rho=28,\beta=8/3$. 
We integrate the Lorenz system starting from $(1,0,0)$ with time interval 0.005 over 2000 time steps. Then we sub-sample the trajectory a time interval of 0.02. The resulting trajectory is our \emph{raw data}. 
Unlike the previous example of Duffing oscillator, there are no separated invariant sets to exploit here but only a $\mathbb{Z}_2$-symmetry in the dynamics itself with the following matrix representation of the group action, 
\begin{equation}
    \gamma_\rho = \begin{bmatrix}
        -1 & 0 & 0 \\
        0 & -1 & 0 \\
        0 & 0 & 1 
    \end{bmatrix}.
\end{equation}
In this case, we can still augment any given trajectory data $\mathcal{D} = \{x_n\}_{n=1}^{M}$ with $\gamma_\rho$. As a result, we have $\mathcal{D}_{\text{aug}} = \{ \{x_n\}_{n=1}^{M}, \{\gamma_\rho x_n\}_{n=1}^{M}\}$, which immediately doubles the size of data. Furthermore, to make an interesting comparison, we cut the original trajectory in the middle and use $\gamma_\rho$ to augment one half of original trajectory. To summarize, we can train a separate vanilla EDMD on each of the three different datasets: (1) original data (499 points), (2) symmetry augmented data (998 points), (3) symmetry augmented on half of the original data (500 points). To generate the unseen test data, we continue to integrate the system for another 1000 time steps from the last time step of original data. 

\begin{figure*}[htbp]
 \centering
 \includegraphics[width=0.6\linewidth]{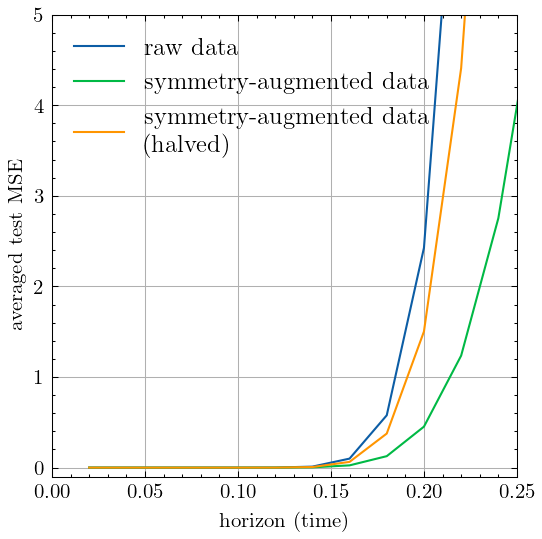}
 \caption{Generalization performance of vanilla EDMD on three different datasets from the chaotic Lorenz attractor. Since this is a chaotic system, long-term point-wise mean-squared error is less meaningful. Thus we compute the average mean-squared error conditioned in the horizon of future state prediction. Although all models work well in the short time, both two EDMD models trained on symmetry-augmented data performs better in the longer horizon.}
 \label{fig:result_lorenz_2}
\end{figure*}

As shown in Fig.~\ref{fig:result_lorenz_2}, it is not surprising that that EDMD trained on symmetry-augmented data (light green) performs the best, since it uses more training data. However, it is interesting that the EDMD trained on halved symmetry-augmented data (yellow), still generalizes better than the one trained on raw data with the same level of data points.

% [PLAN] Here we will show sampling a single attractor then exploit the known symmetry to extend to the entire domain of attraction is better than using a generic or neural network approach to learn the entire phase space. 

% [01/18/2024.] I found no benefits of using symmetric formulation. Only a 3rd order polynomial nailed the system. Using the whole data is the same good as using symmetric formulation, and sometimes even better. 

% A good question: using naive data augmentation or using symmetric formulation, which is better? It turns out that symmetric formulation is much better. 

% [01/18/2024] I found if I try using radial basis function with properly distributed centers. I can show a statistical trend of the benefit in using symmetric formulation. It turns out to be great since it only uses half of the data, but managed to perform better. 

% We will show this is much data efficient and no need to learn indicator function for a prediction task. 

% [01/18/2024] What I found new is, previous research only discussed the case where attractors are disjoint. In fact, if it is not disjoint invariant set but. 

\section{Conclusions}
\label{sec:conclusion}

This work investigates  the role of lifting and linear reconstruction of nonlinear systems using multiple disjoint invariant sets.
Specifically, we provided an explanation on the mechanism for learning a Koopman invariant subspace of an nonlinear dynamical system with multiple disjoint invariant sets that span the system state in a \emph{weak} sense by introducing discontinuous observables. 
Furthermore, we discussed the role of symmetry  on the aforementioned mechanism. Under mild conditions, we derived a general formula for Koopman mode decomposition of nonlinear system with multiple disjoint invariant sets. Specifically, we studied the unforced Duffing oscillator and showed that one can further reduce the lifted dimensions by exploiting symmetry transformation. 
Lastly, motivated by our analysis, we performed numerical experiments that demonstrate the effectiveness of symmetry-constrained EDMD on the unforced Duffing oscillator, and the benefits of augmenting data using symmetries for a chaotic Lorenz system.

\section*{Acknowledgements}
% This work was supported by AFOSR  under the grant FA9550-17-1-0195.
This work was supported by AFOSR  under the grant FA9550-17-1-0195. In addition, the authors appreciate helpful comments from Igor Mezi\'{c} and Matthew Kvalheim.
  
\bibliographystyle{siamplain}
\bibliography{panlab}

\end{document}